\documentclass[11pt]{amsart}
\usepackage{amssymb,pstcol,pst-plot}

\parskip=\smallskipamount

\newtheorem{theorem}{Theorem}[section]
\newtheorem{lemma}[theorem]{Lemma}
\newtheorem{corollary}[theorem]{Corollary}
\newtheorem{proposition}[theorem]{Proposition}

\theoremstyle{definition}
\newtheorem{definition}[theorem]{Definition}
\newtheorem{example}[theorem]{Example}

\newtheorem{problem}[theorem]{Problem}

\definecolor{myblue}{rgb}{0.66,0.78,1.00}

\newcommand\hra{\hookrightarrow}

\newcommand{\cC}{\mathcal{C}}

\newcommand{\B}{\mathbb{B}}
\newcommand{\C}{\mathbb{C}}
\newcommand{\D}{\mathbb{D}}

\newcommand{\N}{\mathbb{N}}

\renewcommand{\P}{\mathbb{P}}
\newcommand{\R}{\mathbb{R}}

\newcommand\wt{\widetilde}
\newcommand{\Aut}{\mathop{{\rm Aut}}}

\def\dibar{\overline\partial}
\def\bs{\backslash}
\def\e{\epsilon}

\numberwithin{equation}{section}

\begin{document}

\title{}
\thanks{\textit{E-mail adresses:}  franc.forstneric@fmf.uni-lj.si, erlendfw@math.uio.no}
\thanks{F.\ Forstneri\v c supported by grants P1-0291 and J1-6173, Republic of Slovenia, and
by Schweizerische Nationalfonds grant 200021-116165/1.}
\thanks{E.\ F.\ Wold supported by Schweizerische Nationalfonds grant 200021-116165/1.}

%
%    General info
%
\date{August 31, 2008}

\LARGE
\centerline{\bf Bordered Riemann surfaces in $\mathbb C^2$}

\vskip 1cm

\large

\centerline{Franc Forstneri\v c\footnote{Corresponding author}$^{,a}$ and Erlend  Forn\ae ss Wold$^{b}$}

\medskip
\medskip

\tiny

\centerline{$^{a}$ \textit{Faculty of Mathematics and Physics, University of Ljubljana and
Institute of}}
\centerline{\textit{Mathematics, Physics and Mechanics, Jadranska 19, 1000 Ljubljana, Slovenia}}

\centerline{$^b$ \textit{Matematisk Institutt, Universitetet i Oslo,
Postboks 1053 Blindern, 0316 Oslo, Norway}}

\rm

\vskip 7mm

\tiny

\centerline{\bf ABSTRACT\rm}
We prove that the interior of any compact complex
curve with smooth boundary in $\C^2$ admits a 
proper holomorphic embedding into $\C^2$. 
In particular, if $D$ is a bordered Riemann surface whose closure admits
a holomorphic embedding into $\C^2$, then $D$ admits a
proper holomorphic embedding into $\C^2$.
\bigskip

\centerline{\bf R\'{E}SUM\'{E}\rm}
On montre que l'interieur d'une courbe complexe compacte avec
bord lisse dans $\mathbb C^2$ admet un plongement holomorphe propre dans
$\mathbb C^2$. En particulier, si D est une surface de Riemann avec bord
dont la fermeture admet un plongement holomorphe dans $\mathbb C^2$,
alors D admet un plongement holomorphe propre dans $\mathbb C^2$.

\normalsize
\medskip
\rightline{\it To Edgar Lee Stout on the occasion of his 70th birthday}

\maketitle

\section{Introduction}
\label{Intro}
It is an old problem whether every open Riemann surface is biholomorphically
equivalent to a topo\-lo\-gi\-cally closed smooth complex curve in
$\C^2$. Equivalently, {\em does every open Riemann surface embed properly
holomorphically in $\C^2$~?} (See Bell and Narasimhan
\cite[Conjecture 3.7, p.\ 20]{BN}.) Such $D$ always embeds in $\C^3$
and immerses in $\C^2$ \cite{Bishop, Narasimhan, Remmert}.

A {\em bordered Riemann surface} is a compact one dimensional
complex manifold, $\overline D$, not necessarily connected,
with smooth boundary $bD$ consisting of finitely many
closed Jordan curves. 
The embedding problem naturally decouples in the following two problems:
\begin{itemize}
\item[(a)] find a (non-proper) holomorphic embedding $f\colon \overline D\hra \C^2$;
\item[(b)] push the boundary of the compact complex curve 
$\overline \Sigma=f(\overline D)\subset \C^2$ to infinity 
without introducing any double points.
\end{itemize}

In this paper we give a complete solution to the second problem,
also for curves with interior singularities.
The following is our main result.

\begin{theorem}
\label{curves}
If $\overline \Sigma$ is a (possibly reducible) compact complex curve 
in $\C^2$ with boundary $b\Sigma$ of class $\cC^r$ 
for some $r>1$, then the inclusion 
map $\iota\colon \Sigma=\overline \Sigma\bs b\Sigma \hra \C^2$
can be approximated, uniformly on compacts in $\Sigma$, 
by proper holomorphic embeddings $\Sigma\hra\C^2$. 
In particular, a smoothly bounded relatively compact domain $\Sigma$ 
in an affine complex curve $A\subset \C^2$ admits
a proper holomorphic embedding in $\C^2$.
\end{theorem}

The precise assumption on $\overline\Sigma$ 
is that locally near each boundary point $p\in b\Sigma$ it is a one 
dimensional complex manifold with boundary of class $\cC^r$, while
the interior $\Sigma$ is a pure one dimensional analytic subvariety
with at most finitely many  singularities.

Theorem \ref{curves} is proved in \S \ref{proof}.
It includes the following result which gives an affirmative 
answer to the problem posed in \cite[p.\ 686]{CF} 
and which contains all known results on embedding
bordered Riemann surfaces properly holomorphically
in $\C^2$. For Riemann surfaces with punctures see 
also Theorem \ref{puncture} and Corollary \ref{cor-punctures} below. 

%
%
%  MAIN THEOREM
%
%
\begin{corollary}
\label{main}
Assume that $\overline D$ is a bordered Riemann surface with
$\cC^r$ boundary for some $r>1$ and that $f \colon \overline D\hra\C^2$ is 
a $\cC^1$ embedding which is holomorphic in $D$.
Then $f$ can be approximated, uniformly on compacts in $D$,
by proper holomorphic embeddings $D\hra\C^2$. 
\end{corollary}

\begin{proof}
A bordered Riemann surface $\overline D$ 
with $\cC^r$ boundary is biholomorphic to a relatively compact 
smoothly bounded domain $D'$ in an open Riemann surface $R$.
Furthermore, if $r$ is a noninteger then any biholomorphic map 
$D\to D'$ extends to a $\cC^r$ diffeomorphism $\overline D\to \overline D'$.
(See the discussion and the references in Section \ref{Teichmuller}.) 
Hence we may assume that $D$ is a relatively compact domain with smooth 
boundary in a Riemann surface $R$. 

By Mergelyan's theorem we can approximate $f$ 
in the $\cC^1$ topology on $\overline D$
by a holomorphic map $\wt f\colon U\to\C^2$ from an open 
set $U\subset R$ containing $\overline D$. If the approximation
is sufficiently close and $U$ is chosen sufficiently small,
then $\wt f$ is a holomorphic embedding of $U$ onto 
a locally closed embedded complex curve without singularities
$A=\wt f(U)$ in $\C^2$. It remains to apply Theorem \ref{curves}
to the complex curve with smooth boundary 
$\overline \Sigma=\wt f(\overline D)$. 
\end{proof}

Corollary \ref{main}, together with the main result
of \cite{KLW}, implies the following result on embeddings
with interpolation on a discerete set.

\begin{corollary}
\label{cor1}
Let $D$ be a bordered Riemann surface satisfying the
hypothesis of Corollary \ref{main}.
Given discrete sequences of points $\{a_j\}\subset D$
and $\{b_j\}\subset \C^2$ without repetitions, there is a proper
holomorphic embedding $\varphi\colon D\hookrightarrow\C^2$
such that $\varphi(a_j)=b_j$ for $j=1,2,\ldots$.
\end{corollary}

In the remainder of this introduction we summarize the main
earlier results on embedding Riemann surfaces in $\C^2$,
and we give a few examples.

An open Riemann surface is the same thing as a one dimensional
Stein manifold. By the classical results (see \cite{Bishop, Narasimhan, Remmert}) 
every open Riemann surface embeds  properly holomorphically in $\C^3$,
and it immerses properly holomorphically in $\C^2$. 
According to Eliashberg and Gromov \cite{EG} and 
Sch\"urmann \cite{Sch}, a Stein manifold of dimension $n>1$ admits 
a proper holomorphic embedding in $\C^N$ with $N=\left[\frac{3n}{2}\right] + 1$.
For $n=1$ this would predict that each open Riemann surface 
embeds properly into $\C^2$, but the proof in the mentioned
papers breaks down in this lowest dimensional case.
The main problem is that self-intersections (double points)
of an immersed complex curve in $\C^2$ are stable under deformations.

The oldest results for embedding Riemann surfaces 
in $\C^2$ are due to Kasahara and Nishino \cite{Stehle}
(for the disc $\D=\{z\in \C\colon |z|<1\}$), Laufer \cite{Laufer}
(for annuli $A=\{r_1 <|z|<r_2\}$),
and Alexander \cite{Alexander} (for $\D$ and $\D \bs \{0\}$);
these were essentially the only known results at the
time of the survey by Bell and Narasimhan \cite{BN}.
In 1995, J.\ Globevnik and B.\ Stens\o nes 
proved that every finitely connected planar domain $D\subset \C$
without isolated boundary points embeds properly holomorphically
into $\C^2$ (see \cite{GS} and also \cite{CF,CG}).

Considerably more general results were obtained  
by the second author in recent papers \cite{Wold1,Wold2,Wold3}. 
In \cite{Wold2}, Corollary \ref{main} was 
proved under the additional assumption that each boundary curve 
$C_j$ of the image $\overline \Sigma=f(\overline D)$ 
contains an {\em exposed point}
$p_j =(p^1_j,p^2_j)$, meaning that the vertical line 
$\{p^1_j\}\times \C$ intersects the curve 
$\overline \Sigma$ only at $p_j$ and the intersection is transverse. 
(See Definition \ref{exposed} and Theorem \ref{Wold} below.)
By applying a shear $g(z,w) = \bigl(z,w+h(z)\bigr)$, where $h$
is a suitably chosen rational function with simple poles at the
points $p^1_j$, the exposed points $p_j$ are blown off to infinity
and we obtain an unbounded embedded complex curve 
$\overline X = g(\overline \Sigma \bs\{p_1,\ldots,p_m\}) \subset\C^2$ 
whose boundary $bX$ consists of the arcs $\lambda_j=g(C_j\bs \{p_j\})$
stretching to infinity.
By a sequence of holomorphic automorphisms $\Phi_n$ of $\C^2$ we then
push $bX$ to infinity, insuring that the sequence
converges to a Fatou-Bieberbach map 
$\Phi=\lim_{n\to\infty} \Phi_n\colon \Omega\to \C^2$
such that $X\subset \Omega$ and $bX\subset b\Omega$.
The restriction $\Phi|_X\colon X\hra \C^2$ is then a proper 
holomorphic embedding of $X$ (that is biholomorphic to $D$)
into $\C^2$. The relevant results on automorphisms of $\C^2$ 
come from the papers \cite{AL,BF,FR}. A list of Riemann surfaces 
that can be embedded in $\C^2$ by Wold's method 
can be found in \cite[Theorem 1]{KLW}.

We prove Theorem \ref{curves} in \S\ref{proof} by first
modifying $\overline \Sigma$ to a biholomorphically
equivalent complex curve which contains an exposed point 
in each boundary component (see Theorem \ref{exposing}); 
this is the main new technical result of this paper. 
The proof is then completed by Wold's method as in \cite{Wold2}
(see Theorem \ref{Wold}). 

A main difference between our construction in this paper and 
those of Globevnik and Stens\o nes \cite{GS} (for
planar domains) and Wold \cite{Wold3} (for domains in
tori) is that {\em the conformal structure on $D$ does not change
during the construction}, and hence we do not need the
uniformization theory in order to complete the proof. 

In \S\ref{Teichmuller} 
we  sketch another  possible proof of Corollary \ref{main} 
by using Teich\-m\"uller spaces of bordered Riemann surfaces.

%
%
% EXAMPLE: complements of intersections with a projective line.
%
%
\begin{example}
\label{example1}
Let $R$ be a smooth closed algebraic curve in the projective plane
$\P^2$. If $U_1,...,U_k$ are pairwise disjoint smoothly bounded 
discs in $R$ whose union contains the intersection of $R$ with 
a projective line $\P^1\subset \P^2$, then the bordered Riemann surface 
$D= R\bs {\bigcup_{i=1}^k} \overline U_i\subset \P^2\bs \P^1=\C^2$ embeds 
properly holomorphically into $\C^2$ according to Corollary \ref{main}.
In particular, since every one dimensional complex torus embeds as a 
smooth cubic curve in $\P^2$, with a given point going to the line at infinity,
we see that any finitely connected subset without isolated boundary points
in a torus embeds properly into $\C^2$. (This is the main
theorem in \cite{Wold3}.)
\qed \end{example}

%
%
%  EXAMPLE: HYPERELLIPTIC SURFACES
%
%
\begin{example}
\label{example2}
A compact Riemann surface $R$ is called {\em hyperelliptic}
if it admits a meromorphic function of degree two,
i.e., a two-sheeted branched holomorphic covering $R\to\P^1$
(see \cite[p.\ 247]{Griffiths-Harris}).
Such $R$ is the normalization of a complex curve
in $\P^2$ given by $w^2=\Pi_{j=1}^k (z-z_j)$ for some choice
of points $z_1,\ldots,z_k\in \C$ (see \cite{FK}). A bordered
Riemann surface $D$ is hyperelliptic if its double 
is hyperelliptic. (The double of $D$ is 
obtained by gluing two copies of $\overline D$, 
the second one with the conjugate conformal structure, along their 
boundaries; see \cite[p.\ 217]{Spr}.) 
Such $\overline D$ admits a holomorphic embedding into 
the closed bidisc $\overline \D^2 \subset\C^2$ by a pair of inner functions
mapping $bD$ to the torus $(b\,\D)^2$
(see Rudin \cite{Rudin} and Gouma \cite{Gouma}).
Hence Corollary \ref{main} implies that 
{\em every hyperelliptic bordered Riemann surface $D$, 
and also every smoothly bounded domain in such $D$,
admits a proper holomorphic embedding in $\C^2$}.
The first statement is  known \cite[Corollary 1.3]{CF},
but the second one is new.
\qed \end{example}

For the general theory of Riemann surfaces
see \cite{Ahlfors-Sario,FK,Griffiths-Harris,Spr}, 
and for the theory of Stein manifolds see \cite{GR}.

%
%
%  SECTION 2
%
%
\section{Construction of a conformal diffeomorphism}
\label{conformal-diff}
The main result of this section is Theorem \ref{conformal-map} 
which is one of our main tools in the proof of Theorem \ref{curves}.

We begin with a lemma on conformal mappings.
Denote by $\D$ the open unit disc in $\C$, 
and by $r\D$ the disc of radius $r>0$. 

\begin{lemma}
\label{lemma1} 
Assume that $R$ is a connected open Riemann
surface, $G\Subset R$ is an open simply connected domain with
smooth boundary, $V'\Subset  V'' \subset R$ are small
neighborhoods of a boundary point $a\in bG$, $b$ is a point in $R
\bs\overline G$, $\gamma$ is a smooth Jordan arc with endpoints
$a$ and $b$ such that $\gamma\cap\overline G=\{a\}$ and the
tangent lines to $\gamma$ and $b\,G$ at the point $a$ are
transverse, and $V$ is a neighborhood of $\gamma$. Then there
exists a sequence of smooth diffeomorphisms
$\psi_{n}\colon\overline G \to \psi_n(\overline G) \subset R$ that
are conformal on $G$ and satisfy the following properties for
$n=1,2,\ldots$:
\begin{itemize}
\item[(i)]   $\psi_{n}\rightarrow\mathrm{id}$ locally uniformly on
$G$ as $n\rightarrow\infty$,
\item[(ii)]  $\psi_{n}(a)=b$, and
\item[(iii)] $\psi_{n}(\overline V'\cap\overline G)\subset V''\cup V$.
\end{itemize}
\end{lemma}

\begin{proof}
Since $\overline G\cup\gamma$ admits a simply connected neighborhood in $R$,
and since we are going to construct maps with images near $\overline G\cup\gamma$,
we might as well assume that we are working in the complex plane,
that $a$ is the origin, and that the strictly positive real axis
lies in the complement of $\overline G$ near the origin.

For each $n\in\N$ let $l_n$ denote the line segment between $0$ and $\frac{1}{n}$
in $\R\subset\C$. Let $\wt V$ be a neighborhood of the origin with $V'\Subset\wt V\Subset V''$.
By approximation there are neighborhoods $U_{n}$ of $\overline G\cup l_{n}$ and holomorphic
injections $f_{n}\colon U_{n}\rightarrow\C$ such that the following hold
for all $n\in\N$:
\begin{itemize}
\item[(1)] $f_{n}\rightarrow\mathrm{id}$ uniformly on $\overline G$ as $n\to\infty$,
\item[(2)] $f_{n}(l_{n})$ approximates $\gamma$, with $f_{n}(\frac{1}{n})=b$ and
$f_{n}(l_{n})\subset V$, and
\item[(3)] $f_{n}(\overline G\cap\wt V)\subset V''$.
\end{itemize}
Of course property (3) is a consequence of (1) for large enough $n$.
For the details of this approximation argument
see e.g.\ \cite{Stolz} or \cite[Theorem 3.2]{HW}
(for $\cC^0$ approximation), and \cite[Theorem 3.2]{FF:submersions}
for the general case with smooth approximation on $l_n$.

For small positive numbers $\epsilon$ we let $\Omega_{\epsilon}$ denote
domains obtained by adding an $\epsilon$-strip around $l_{n}$ to $G$,
containing the point $\frac{1}{n}$ in the boundary $b\,\Omega_\e$.
We smoothen corners to obtain smoothly bounded domains.
We let $R_{\epsilon}$ denote
the part of $\Omega_{\epsilon}$ that is not in $G$.

Choose a sequence $\epsilon_{n}\searrow 0$ such that
$\overline \Omega_{\e_n} \subset U_{n}$ for each $n\in\N$.
Write $\Omega_n=\Omega_{\e_n}$ and $R_n=R_{\e_n}$.
By choosing the $\epsilon_{n}$'s small enough we get that
\begin{itemize}
\item[(4)]  $f_{n}(R_n)\subset V$ for each $n\in\N$.
\end{itemize}
Next we choose a point $p\in G$ and a sequence of conformal maps
$g_{n} \colon G\rightarrow\Omega_{n}$ such that
$g_{n}(p)=p$ and $g_{n}'(p)>0$ for $n=1,2,\ldots$.  
Since our domains are smoothly
bounded, the map $g_n$ extends to a smooth diffeomorphism
of $\overline G$ onto $\overline \Omega_n$.
Furthermore, since the domains $\overline \Omega_{n}$ 
converge to $\overline G$ as $n\to\infty$, we conclude by 
Rado's theorem (see e.g.\ \cite[Corollary 2.4, p.\ 22]{Pom} 
or \cite[Theorem 2, page 59]{Goluzin}) that
\begin{itemize}
\item[(5)] $g_{n}\rightarrow\mathrm{id}$ uniformly on $\overline G$ as $n\to\infty$.
\end{itemize}
Hence for $n$ large enough we have that
$g_{n}(\overline{V'}\cap \overline{G})\subset(\wt V\cap\overline G)\cup R_n$.
% (keeping in mind that $g_{n}$ is injective as well as close to the identity).
Combining this with (3) and (4) we see that
$f_{n}\circ g_{n}(\overline{V'}\cap\overline G)\subset V''\cup V$
if $n$ is large enough.  Hence, by defining $\psi_{n}:=f_{n}\circ g_{n}$
we get property (iii) for all large $n$, 
and we clearly also get property (i).

To see that property (ii) holds, let $a_{n}\in b\,G$ denote the point that
$g_{n}$ sends to $\frac{1}{n} \in b\Omega_n$. By (5) the sequence $a_{n}$ has to
converge to the origin, and so there is a sequence of conformal
automorphisms $\varphi_{n}$ of $G$ fixing the point $p$, sending the origin to $a_{n}$,
with $\varphi_{n}\rightarrow\mathrm{id}$ uniformly on $\overline G$.  
Replacing the maps $g_{n}$ by $g_{n}\circ\varphi_{n}$ in 
the above argument also gives (ii).
\end{proof}

In the remainder of this section, $R$ denotes a Riemann surface
without boundary and $D$ is a relatively compact,
smoothly bounded domain with nonempty boundary
in $R$, not necessarily connected.
The following Lemma provides the main inductive step
in the proof of Theorem \ref{exposing}.

\begin{lemma}
\label{main-step} 
Given pairwise distinct points $a_1, a_2,\ldots,a_k\in \overline D$ 
with $a_1 \in bD$, a neighborhood
$U\subset R$ of $a_1$, a point $b\in R\bs \overline D$ in the same
connected component of $R\bs D$ as $a_1$, and a positive
integer $N\in\N$, there is a smooth diffeomorphism $\phi\colon
\overline D\to \overline D'\subset R$ satisfying the following:
\begin{enumerate}
\item $\phi\colon D\to D'$ is biholomorphic,
\item $\phi(a_1)= b$,
\item $\phi$ is tangent to the identity map to order $N$ at each
of the points $a_2,\ldots,a_k$, and
\item $\phi$ is as close as desired to the identity map on $\overline D\bs U$
in the smooth topology on the space of maps.
\end{enumerate}
\end{lemma}

\begin{proof}
We may assume that $N>2$.  Choose a smooth
embedded Jordan arc $\gamma \subset R$ with the endpoints $a_1$ and
$b$ such that $\gamma\cap \overline D=\{a_1\}$, and the tangent line
to $\gamma$ at $a_1$ intersects the tangent line to $bD$ at $a_1$
transversely. Then $\gamma$ has an open, connected and simply
connected neighborhood $W\subset R$ that is conformally equivalent
to a bounded domain (a disc) in $\C$. Let $z$ denote the
corresponding holomorphic coordinate on $W$, chosen such that
$z(a_1)=0$. By shrinking the neighborhood $U$ of the point $a_1$ we
may assume that $\overline U\subset W$, that $\overline U$ does
not contain any of the points $a_2,...,a_k$, and that 
$z(U)=r\D \subset\C$ for some $r>0$. Choose a number $r' \in (0,r)$ and let
$U'\subset U$ be chosen such that $z(U')=r'\D$.

Choose a connected and simply connected domain $G\subset W$ with
smooth boundary, with a defining function $\rho$ such that
$G=\{\rho<0\}$ and $d\rho\ne 0$ on $b\,G$, satisfying the following
properties (see Figure \ref{Fig2}): 
\begin{itemize}
\item[(i)]   $\overline D\cap U \subset G\cup \{a_1\}$,
\item[(ii)]  $-\rho(z) \ge \mathrm{const}\cdotp \mathrm{dist}(z,a_1)^2$
for points $z\in bD$ close to $a_1$,  and
\item[(iii)]  $\gamma\cap \overline G = \{a_1\}$.
\end{itemize}
Property (iii) can be achieved since the arc $\gamma$
is transverse to $bD$ at $a_1$.

%
%
%  Fig. 2: %
%

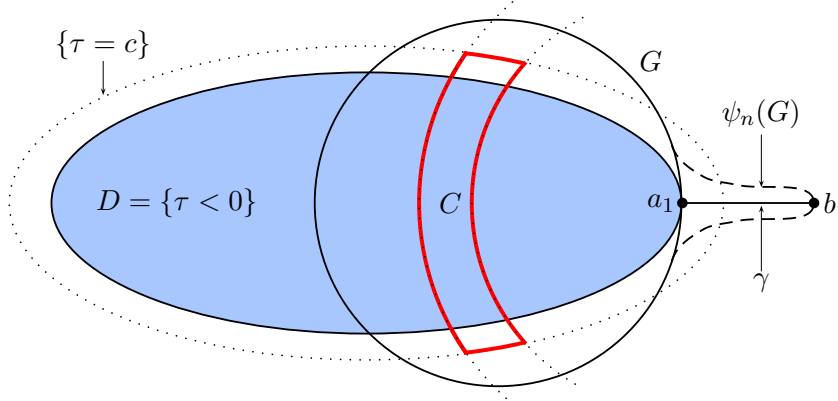
\begin{figure}[ht]
\psset{unit=0.7cm, linewidth=0.7pt}

\begin{pspicture}(-8,-3.5)(10,4)

\psarc[linewidth=1.5pt,linecolor=red](6,0){5}{145}{215}

%
% Manjsa elipsa $D$ z modro notranjostjo
%
\psellipse[fillstyle=solid,fillcolor=myblue](0,0)(6,2.5)
\rput(-3.6,0){$D=\{\tau<0\}$}
%
% Vecja pikcasta elipsa z oznako
%
\psellipse[linestyle=dotted](0,0)(6.8,3)
\rput(-5,3){$\{\tau = c\}$}
\psline[linewidth=0.2pt]{<-}(-5,2.05)(-5,2.7)

%
% Krog za domeno G
%
\pscircle(2.5,0){3.5}
\rput(5.4,2.7){$G$}

%
%  Lok gamma s krajisci
%
\psline(6,0)(8.5,0)
\psdot[dotsize=4pt](6,0)
\psdot[dotsize=4pt](8.5,0)
%
%  oznake ob loku
%
\rput(7.5,-1.5){$\gamma$}
\psline[linewidth=0.2pt]{<-}(7.5,-0.05)(7.5,-1.3)

\rput(5.6,0){$a_1$}
\rput(8.8,0){$b$}

%
%  mnozica C - dva rdeca loka
%
\psarc[linewidth=1.5pt,linecolor=red](6,0){5}{145}{215}
\psarc[linewidth=1.5pt,linecolor=red](6,0){4}{138}{222}

%
%  in se dva vecja pikcasta loka: U in U'
%
\psarc[linestyle=dotted](6,0){5}{130}{230}
\psarc[linestyle=dotted](6,0){4}{118}{242}
%
%  obroba mnozice C - zgoraj in spodaj
%
\psecurve[linewidth=1.5pt,linecolor=red](0,3)(1.85,2.85)(3,2.65)(5,2)
\psecurve[linewidth=1.5pt,linecolor=red](0,-3)(1.85,-2.85)(3,-2.65)(5,-2)
%
% Oznaka C
%
\rput(1.6,0){$C$}

%
% Domena \psi(G)=W
%
\psecurve[linestyle=dashed]
(6,2)(5.8,1.1)(6.4,0.5) (7,0.35) (8.5,0)
(7,-0.35)(6.4,-0.5)(5.8,-1.1)(6,-2)

\rput(7.5,1.7){$\psi_n(G)$}
\psline[linewidth=0.2pt]{<-}(7.5,0.3)(7.5,1.3)

\end{pspicture}
\caption{The domains $D$ and $G$}
\label{Fig2}
\end{figure}

Choose a smooth defining function $\tau$ for the domain
$D$ such that $D=\{\tau<0\}$ and $d\tau\ne 0$ on $bD=\{\tau=0\}$.
Choose a small number $c>0$ and let
\[
    A=\{\tau\le c\} \bs U',     \quad
    B=\{\tau\le c\} \cap \overline U, \quad
    C=\{\tau\le c\} \cap (\overline U\bs U').
\]
By choosing $c>0$ small enough we insure that $C$
is a compact set contained in $G$ (see Figure \ref{Fig2}),
and we have
\[
    A\cup B = \{\tau\le c\}, \quad A\cap B = C.
\]
On Figure \ref{Fig2}, the set $C$ is bounded by the two circular arcs
(left and right) and by the two arcs in the larger dotted ellipse
representing the level set $\{\tau= c\}$.
The set $A$ is the part of the filled dotted ellipse
lying on the left hand side of the right boundary arc of $C$, and
$B$ is the part of the filled dotted ellipse on the
right hand side of the left boundary arc of $C$.

Choose small open neighborhoods $V'\Subset V''$ of the point $a$
such that $\overline V''$ is contained in the interior of the set
$B\bs A$, and choose a small neighborhood $V$ of $\gamma$ such
that $\overline V\cap\overline{(A\bs B)}\cap\overline
D=\emptyset$. Let $\psi_n\colon \overline G\to \psi_n(\overline G)$ 
be a sequence of conformal maps furnished by Lemma
\ref{lemma1}, satisfying the properties of that lemma
with respect to the sets $V,V',V''$. 
Recall that the compact set $C$ is contained in $G$.
Choose an open set $C'\Subset G$ containing $C$. On $C'$ we write
\[
    \mathrm{id} = \psi_n\circ \gamma_n,\quad \gamma_n=\psi_n^{-1}.
\]
As $n\to +\infty$, $\psi_n$ converges to the identity uniformly
on $C'$, and hence also in the smooth topology (by the Cauchy estimates).
The same is then true for its inverse $\gamma_n$ on a slightly
smaller neighborhood of $C$.

We are now in position to apply \cite[Theorem 4.1]{F:Acta}
to the map $\gamma_n$. For every sufficiently large $n\in\N$,
the cited theorem furnishes a decomposition  % of the form
\[
    \gamma_n \circ \alpha_n =\beta_n \quad {\rm near\ } C,
\]
where $\alpha_n$ is a small holomorphic perturbation of the
identity map on a fixed neighborhood of $A$
(independent of $n$) that is tangent to the identity to
order $N$ at each of the points $a_2,\ldots,a_k$,
and $\beta_n$ is a small holomorphic perturbation of the
identity map on a neighborhood of $B$ that is tangent to the identity
to order $N$ at the point $a_1$.
The closeness of $\alpha_n$ (resp.\ of $\beta_n$)
to the identity in any $\cC^r$ norm on $A$ (resp.\ on $B$)
can be estimated by the closeness of $\psi_n$ to the identity
on $C'$. (This Cartan-type decomposition lemma
for biholomorphic maps close to the identity is one of the
most essential results used in our construction. Its proof in \cite{F:Acta}
applies to Cartan pairs in an arbitrary Stein manifold.)

By combining the above two displays we obtain
\[
      \alpha_n=\psi_n\circ\beta_n \quad{\rm near\ } C.
\]
If the approximations are sufficiently close (which holds for
$n$ large enough) then the two sides, restricted to
$A\cap\overline D$ (resp.\ to $B\cap \overline D$), define a
diffeomorphism $\phi_n \colon \overline D\to \phi_n(\overline D) \subset R$
that is holomorphic in $D$ and such that
\begin{itemize}
\item $\phi_n(a_1)=b$,
\item $\phi_n$ is tangent to the identity map
to order $N$ at each of the points $a_2,\ldots,a_k$, and
\item $\phi_n$ converges to the identity map
uniformly on $\overline D\bs U$ as $n\to +\infty$.
\end{itemize}

Indeed, both sides $\alpha_n$ and $\psi_n\circ\beta_n$
satisfy the stated properties on their respective domain.
For $\alpha_n$ this is clear from the construction.
For $\beta_n$ we need a more precise argument to see that
it maps $B\cap \overline D$ into $G\cup\{a\}$ for sufficiently large
$n\in\N$. By the construction, its Taylor expansion in a
local holomorphic coordinate $z$ near $a_1$, with $z(a_1)=0$, equals
\[
    \beta_n(z)=z + M_n z^N + O(z^{N+1}).
\]
The size of the constant $M_n$, and of the remainder term,
can be estimated (using the Cauchy estimates)
by $\mathrm{dist}(\beta_n,\mathrm{id})$ on $B$,
and hence by $\mathrm{dist}(\psi_n,\mathrm{id})$ on the set $C'$.
Since $G$ osculates $D$ from the outside to the second order
at the point $a_1$ (see property (ii) above), 
it follows that for a sufficiently small
neighborhood $U_1$ of the point $a_1$ and for all large enough
$n\in\N$ we have
\begin{equation}
\label{betan}
    \beta_n(\overline D\cap U_1) \subset (G\cup\{a_1\})\cap V'.
\end{equation}
On the complement $(B\cap \overline{D}) \bs U_1$,
$\beta_n$ is close to the identity for large $n$,
and hence it maps this set into a fixed compact set in $G$.
Thus the composition $\psi_n\circ\beta_n$ is well defined on
$B\cap \overline D$ and it satisfies the stated properties.

It is also easily seen that $\phi_n$ is injective if $n$ is large
enough. Indeed, each of the two expressions defining $\phi_n$ on
$A\cap \overline D$ (resp.\ on $B\cap \overline D$) is injective by the
construction, and hence it suffices to verify
that no point from $(A\bs B)\cap \overline D$ can
get identified with a point from $(B\bs A)\cap \overline D$
under $\phi_n$. By the construction, the points from the first
set remain nearby since $\alpha_n$ is close to the identity.
Consider now points $x\in (B\bs A)\cap \overline D$.
If $x\in U_1$ then $\beta_n(x)\in (G\cup\{a_1\})\cap V'$ by (\ref{betan}),
and hence $\psi_n\circ\beta_n(x) \in V''\cup V$
by property (iii) in Lemma \ref{lemma1}.
Since the set $V''\cup V$ is at a positive distance
from $(A\bs B)\cap \overline D$, we see that $\psi_n\circ\beta_n(x)$
cannot coincide with $\alpha_n(x')$ for any point 
$x'\in (A\bs B)\cap \overline D$ provided that $n$ is large enough.
The remaining set $((B\bs A)\cap \overline D) \bs U_1$ is
compactly contained in $B\cap \overline D\cap G$ where
$\psi_n\circ\beta_n$ is close to the identity for large $n$,
and hence no point from this set can get identified 
with a point from $(A\bs B)\cap \overline D$.
\end{proof}

Using Lemma \ref{main-step} inductively we now
prove the following result.

\begin{theorem}
\label{conformal-map}
Assume that $D$ is a relatively compact smoothly bounded
domain in a Riemann surface $R$. Choose finitely many pairwise distinct points
$a_1,\ldots,a_k\in bD$, $b_1,\ldots,b_k \in R\bs \overline D$, and
$c_1,\ldots,c_l\in \overline D\bs\{a_1,\ldots,a_k\}$
such that for each $j=1,\ldots,k$ the points
$a_j$ and $b_j$ belong to the same connected component of $R\bs D$.
For every integer $N\in\N $ there exists a diffeomorphism
$\phi\colon \overline D\to \overline D'$
onto a smoothly bounded domain $D'\subset R$ such that
$\phi\colon D\to D'$ is biholomorphic, $\phi(a_j)= b_j$
for $j=1,\ldots, k$, and $\phi$ is tangent to the identity map to
order $N$ at each point $c_j$. Furthermore, given a neighborhood
$U_j$ of $a_j$ for every $j$, $\phi$ can be chosen as close as
desired to the identity map in the smooth topology on
$\overline D\bs \bigcup_{j=1}^k U_j$.
\end{theorem}

\begin{proof}
By decreasing the neighborhoods $U_j\ni a_j$ we may assume
that their closures are pairwise disjoint and do not contain
any of the points $c_j$. Choose smaller neighborhoods
$U'_j \ni a_j$ with $\overline U'_j\subset U_j$ for 
$j=1,\ldots, k$.

A map $\phi$ with the desired properties will be found as a composition
\[
    \phi= \phi_k\circ\phi_{k-1}\circ\cdots \circ
    \phi_2\circ \phi_1\colon \overline D\to \overline D'.
\]
In the first step, Lemma \ref{main-step} furnishes a
diffeomorphism  $\phi_1\colon \overline D \to \phi_1(\overline D)= \overline D_1$
onto a new domain $\overline D_1 \subset R$ such that
\begin{enumerate}
\item $\phi_1$ is biholomorphic in the interior,
\item $\phi_1(a_1)=b_1$,
\item $\phi_1$ is tangent to the identity
to order $N'=\max\{2,N\}$ at each of the points
$a_2,\ldots, a_k$ and $c_1,\ldots, c_l$, and
\item $\phi_1$ is uniformly close to the identity on $\overline D\bs U'_1$.
\end{enumerate}
Hence the points $b_1=\phi_1(a_1)$, $a_2,\ldots,a_k$ lie on $bD_1$,
and $c_j\in \overline D_1$ for $j=1,\ldots,l$.

In the second step we apply Lemma \ref{main-step},
with $\overline D$ replaced by $\overline D_1=\phi_1(\overline D)$,
to find a diffeomorphism
$\phi_2\colon \overline D_1\to \phi_2(\overline D_1)= \overline D_2$,
holomorphic in the interior and close to the identity map
on $\overline D_1\bs U'_2$, such that $\phi_2(a_2)=b_2$,
$\phi_2$ is tangent to the identity to order $N'$ at the points
$b_1, a_3,\ldots, a_k$ and $c_1,\ldots,c_l$, and $\phi_2$ is close
to the identity map on $\overline D_1 \bs U'_2$.

Continuing inductively, we obtain after $k$ steps a map
$\phi$ satisfying the conclusion of Theorem \ref{conformal-map}
with $D'=D_k$. At the $j$th step of the construction,
the action takes place near the point $a_j\in bD_{j-1}$
that is mapped by $\phi_j$ to the point $b_j\in bD_j=\phi_j(bD_{j-1})$.
In addition, $\phi_j$ is tangent to the identity at the points
$b_1,\ldots, b_{j-1}$, $a_{j+1},\ldots, a_k$ and $c_1,\ldots,c_l$,
and $\phi_j$ is close to the identity map on $\overline D_{j-1}\bs U'_j$.

The final domain $D'=D_k=\phi(D)$ contains the
points $b_1,\ldots,b_k$ in the boundary, while the points
$c_1,\ldots,c_l$ remained fixed during the construction.
The domain $D'$ is very close to $D$ away from a
small neighborhood of each point $a_j$, and at $a_j$
it includes a spike reaching out to $b_j$.
\end{proof}

\section{Normalization and stability of complex curves in $\C^2$}
\label{proof-curves}
In this section we obtain some technical results
that will be used in the proof of Theorem \ref{curves}
in the case of curves with interior singularities.

The first lemma gives a normalization of a complex curve
with smooth boundary by a bordered Riemann surface.

\begin{lemma}
\label{normalize}
Let $\overline \Sigma$ be a compact complex curve with boundary of 
class $\cC^r$ $(r\ge 1)$ in a complex manifold $X$.
There exists a bordered Riemann surface $\overline D$ with 
$\cC^r$ boundary and a $\cC^r$ map $f \colon \overline D\to X$,
with $f(\overline D)=\overline \Sigma$ and $f(bD) = b\Sigma$,   
such that $f$ is a diffeomorphism  near $bD$ and
$f\colon D\to \Sigma$ is a holomorphic normalization of $\Sigma$.
In particular, $f$ is biholomorphic over the regular locus of $\Sigma$.
\end{lemma}

\begin{proof}
To get such $D$ and $f$ we simply normalize each singular point of $\Sigma$
(see e.g.\ \cite[p.\ 70]{Chirka} for  curves without boundaries);
we briefly describe this construction.
The conditions imply that $\Sigma$ has at most finitely many interior
singularities $p_1,\ldots,p_n\in \Sigma$ and no singularities on $b\Sigma$.
Choose a small open set $B_j\subset X$ containing $p_j$  
(in local coordinates at $p_j$, $B_j$ is a small ball) and let 
$\Sigma\cap B_j=\bigcup_{k=1}^{m_j} V_{j,k}$ be a decomposition
into irreducible branches. By choosing $B_j$ sufficiently small
we insure that each $V_{j,k}\bs\{0\}$ is regular, 
$V_{j,k}\cap V_{j,k'}=\{p_j\}$ when $k\ne k'$, 
and the normalization of each $V_{j,k}$ is  a disc in $\C$.
More precisely, there is an injective holomorphic map 
$\psi_{j,k}\colon  \D \to V_{j,k}$, with $\psi_{j,k}(0)=p_j$,
such that $\psi_{j,k} \colon \D\bs\{0\} \to V_{j,k}\bs \{p_j\}$
is biholomorphic. By surgery with $\psi_{j,k}$ we replace 
$V_{j,k}\subset \Sigma$ by the disc $\D$, hence 
$\Sigma\cap B_j$ is replaced by the disjoint union of $m_j$ discs. 
To get $D$ and $f$ it suffices to perform  this construction at every singular 
point $p_j$ of $\Sigma$.
\end{proof}

The following lemma result is a special case of the classical
results on universal denominators (see e.g.\ Whitney \cite{Whitney}).
For completeness we provide a simple proof for curves in $\C^2$
by using a solution to $\dibar$-equation. 
(We thank J.-P.\ Rosay for suggesting such a proof.)

\begin{lemma}
\label{weakly-holo}
Let $V$ be pure one dimensional analytic subvariety near the origin 
in $\C^2$ with $V_{\rm sing}=\{0\}$.
There is an integer $N\in \N$ such that every holomorphic function 
$g$ on $V^*=V\bs\{0\}$ satisfying $|g(z)|\le C|z|^N$ for some
$C>0$ extends across $0$ to a holomorphic function on $V$.
\end{lemma}

\begin{proof}
Write $z=(z_1,z_2)\in\C^2$ and let $\pi_j(z_1,z_2)=z_j$ for $j=1,2$.  
After shrinking $V$ and applying a linear change of coordinates on $\C^2$ 
we may assume that $\pi_1|_V \colon V\to U$ is a branched analytic 
covering over a disc $U=r\D \subset \C$ such that $|z_2|\le |z_1|$ on $V$
(see e.g.\ \cite[\S 6.1]{Chirka}). By shrinking $U$ around $0$
we have for each $z_1\in U^*= U\bs \{0\}$ 
\[
	\pi_2\bigl(V\cap \pi_1^{-1}(z_1)\bigr) 
	= \{b_1(z_1), \ldots, b_m(z_1)\} \subset \C,
\]
where the functions $b_j(z_1)$ are locally holomorphic and
satisfy an estimate 
\[
	|b_j(z_1)-b_k(z_1)|\ge c|z_1|^\nu
\]
for some $\nu\ge 1$, $c>0$ and for all $j\ne k \in\{1,\ldots,m\}$. 
For irreducible $V$ this estimate follows from the Puiseux
series representation, see \cite[p.\ 68]{Chirka}.
In general we use that any two complex curves 
have a finite order of tangency at an isolated intersection point.

Let $W_{j}(z_1)\subset \C$ be the disc of radius 
$\frac{c}{4}|z_1|^\nu$ centered at $b_j(z_1)$,
and let $W'_j(z_1)$ denote the disc of twice that radius; 
hence the larger discs are still pairwise disjoint. Set 
\[
	W=\{(z_1,z_2) \colon z_1\in U^*,\ 
	z_2\in W_j(z_1)\ {\rm for\ some}\ j=1,\ldots,m\}.
\]
Similarly we define the set $W'\supset W$ by taking the union of the discs 
$W'_j(z_1)$ in the fibers. Observe that the distance from a point 
$(z_1,z_2) \in W$ to the complement of $W'$ is comparable to 
$|z_1|^{\nu}$ as $z_1\to 0$, and hence there is a smooth function 
$\chi$ on $U^* \times \C$ with values
in $[0,1]$ such that $\chi=1$ on $W$, $\rm supp \chi \subset W'$,
and $|\dibar \chi(z)| \le c'|z_1|^{-\nu}$ for some $c'>0$.

Choose a holomorphic function $P(z_1,z_2)$
on $U\times\C$ with $V=\{P=0\}$. In fact, $P$ can be chosen as
a Weierstrass polynomial in $z_2$, with coefficients holomorphic
in $z_1\in U$ (see e.g.\ Chirka \cite[p.\ 25]{Chirka}).

Suppose that $g\colon V^*=V\bs\{0\} \to\C$ is a holomorphic function.
We extend $g$ to a holomorphic function on the tube $W'$ by taking
a constant vertical extension on the fiber around each point
$b_j(z_1)$. More precisely, for $(z_1,z_2)\in W'_j(z_1)$
we take $g(z_1,z_2)=g(z_1,b_j(z_1))$. Then $\chi g$ is a well defined
smooth function on $U^*\times\C$ which is holomorphic in $W$
and agrees with the original function $g$ on $V^*$. Note also that
$\dibar (\chi g)= g\dibar  \chi$ is supported in $W'\bs W$
and satisfies $|g\dibar \chi|\le c' |g| |z_1|^{-\nu}$.
We seek a holomorphic extension $G$ of $g$ in the form 
$G= \chi g - uP$; this implies $G=g$ on $V^*$. The holomorphicity condition 
$0=\dibar G = g\dibar\chi - P\dibar u$ is equivalent to 
\[
	\dibar u = \alpha:= \frac{1}{P}\, g \, \dibar \chi = 
	\alpha_1d\bar z_1 + \alpha_2 d\bar z_2.
\]
On the support of $\alpha$ (in $W'\bs W$) we have 
$|P(z_1,z_2)| \ge |z_1|^{\mu}$ for some $\mu>0$. If $N\ge \mu+\nu+2$, 
the estimate $|g(z_1,z_2)|\le |z_1|^N$ for $(z_1,z_2)\in V^*$
implies that $|\alpha(z_1,z_2)|\le c' |z_1|^2$.
For such $\alpha$, the equation $\dibar u=\alpha$ has a solution 
on $U\times\C$ given by
\[
	u(z_1,z_2)=\frac{1}{2\pi i} \int\!\!\!\int_{t\in\C} \frac{\alpha_2(z_1,t)}{t-z_2} 
							\, dt\wedge d\bar t.
\]	
(The integrand is compactly supported for each 
fixed $z_1\in U$, and it vanishes for $z_1=0$.) 
This yields a desired holomorphic extension $G$ of $g$.
\end{proof}

Our next result shows that the biholomorphic type of a holomorphic image 
of a bordered Riemann surface in $\C^2$ does not change under a perturbation 
of the map that is tangent to a sufficiently high order
over every singularity.

\begin{lemma}
\label{invariance}
Let $\overline D$ be a bordered Riemann surface with $\cC^1$ boundary 
and $f\colon \overline D\to \C^2$ be a $\cC^1$ map that is an embedding 
near $bD$ and is holomorphic in $D$, with $f(D)\cap f(bD)=\emptyset$. 
Let $\Sigma=f(D)$, let $p_1,\ldots, p_k \in \Sigma$ be all its singular points, 
and let $\{q_1,\ldots, q_l\}=f^{-1}(\{p_1,\ldots,p_k\}) \subset D$.
Then there exists an integer $N\in\N$ with the following property.
For every $\cC^1$ map $f'\colon \overline D\to \C^2$ which is sufficiently
$\cC^1$ close to $f$, holomorphic in $D$ and tangent to $f$ to order $N$
at each of the points $q_1,\ldots,q_l$, the image $\Sigma'=f'(D)$
is biholomorphically equivalent to $\Sigma$.
\end{lemma}

\begin{proof}
The conditions imply that $f$ and $f'$ are injective holomorphic 
embeddings of $D'=D\bs\{q_1,\ldots,q_l\}$ into $\C^2$, and hence
the map
\[
	\Phi= f'\circ f^{-1} \colon \Sigma\bs \{p_1,\ldots, p_k\} 
	\to \Sigma'\bs \{p_1,\ldots, p_k\}
\]
is biholomorphic. It remains to show that $\Phi$ and $\Phi^{-1}$
extend holomorphically across the singular points $p_j$, 
provided that $f$ and $f'$ are tangent to a sufficiently high
order at all points in $f^{-1}(p_j)\subset D$. 

The problem being local, we fix a point $q=q_j \in D$ and 
let $p =f(q)=f'(q)\in \Sigma_{\rm sing}\subset  \C^2$.
In suitable local holomorphic coordinates we have $q=0\in\C$,
$p =0 \in\C^2$, and $f(\zeta)=(f_1(\zeta),f_2(\zeta))$ is an injective 
local holomorphic map with the only branch point at $\zeta=0$. 
Let $V \subset \Sigma$ be the local image of $f$,
so $V$ is a local irreducible complex curve in $\C^2$ whose only 
singular point is the origin $0\in\C^2$. For $z\in V^*=V\bs \{0\}$ let 
$\zeta(z) = f^{-1}(z)$, a holomorphic function 
on $V^*$. We have $|\zeta(z)|\le |z|^\alpha$ for some $\alpha>0$.
Then $\Phi(z)=f'(\zeta(z))$ for $z\in V^*$.
From  $f'(\zeta)=f(\zeta)+O(\zeta^N)$ we get for $z\in V^*$
\begin{equation}
\label{Phi-g}
	\Phi(z)= f(\zeta(z)) + O(\zeta(z)^N) = z + g(z),
\end{equation}
where $g$ is a holomorphic function on $V^*$ satisfying 
\[
	|g(z)|=O(|\zeta(z)|^N)= O(|z|^{N\alpha}),\quad z\to 0.
\]

The same argument applies to every local irreducible component of 
$\Sigma$ at the singular point $p$. If $N>0$ is sufficiently large 
then the function $g$ in (\ref{Phi-g}), which is 
defined and holomorphic on a deleted neighborhood of $p$ in $\Sigma$,  
extends holomorphically across $p$ by Lemma \ref{weakly-holo}.
It follows that $\Phi$ extends holomorphically to $\Sigma$
for all large $N$. The same  argument applies to  $\Phi^{-1}$, so 
$\Phi \colon \Sigma\to\Sigma'$ is biholomorphic.
\end{proof}

\section{Exposing boundary points}
\label{sec:exposing}
In this secton we prove a result on exposing boundary points of 
complex curves in $\C^2$. Theorem \ref{exposing} below is
a main new technical result of this paper.
It also hold in $\C^n$, with essentially the same proof.

We shall need the following notion introduced in \cite{Wold2}.
Let $\pi\colon \C^2\to \C$ denote a $\C$-linear map onto $\C$;
we may assume that $\pi(z_1,z_2)=z_1$. 

%
%
% EXPOSED POINTS
%
%
\begin{definition}
\label{exposed}
Let $\Sigma \subset \C^2$ be a locally closed complex curve,
possibly with boundary. A point $p=(p_1,p_2)\in \Sigma$ is  {\em exposed}
(with respect to the projection $\pi$) if the complex line
\[
    \Lambda_p= \pi^{-1}(\pi(p)) = \{(p_1,\zeta)\colon \zeta\in \C\}
\]
intersects $\Sigma$ precisely $p$ and the intersection
is transverse: $T_p \Lambda_p \cap T_p \Sigma =\{0\}$.
If $\Sigma=f(R)$, where $R$ is a Riemann surface (with or without
boundary) and $f\colon R\to \C^2$ is a holomorphic map, 
then a point $a\in R$ is said to be $f$-exposed
if the point $p=f(a) \in \Sigma$ is exposed.
\qed \end{definition}

%
%
%  EXPOSING BOUNDARY POINTS
%
%
\begin{theorem}
\label{exposing}
Let $\overline D$ be a bordered Riemann surface with $\cC^r$ boundary 
for some $r>1$. Assume that $f\colon \overline D\to \C^2$ is a $\cC^1$
map which is holomorphic in $D$ and is an embedding near $bD$,
with $f(D)\cap f(bD)=\emptyset$. 
Then $f$ can be approximated, uniformly on compacts in $D$,
by a map $F\colon \overline D \to \C^2$ with the same properties
such that the complex curve $F(D)\subset \C^2$ is biholomorphic to 
the curve $f(D)$,
and such that every boundary curve of $F(\overline D)$ contains an 
exposed point. Furthermore, $F$ can be chosen to agree with $f$ to 
a given finite order at a prescribed finite set of
points $c_1,\ldots,c_l\in D$; if these points are $f$-exposed then
$F$ can be chosen such that they are also $F$-exposed. 
\end{theorem}

\begin{proof}
We begin with a few reductions.

The hypotheses imply that $\overline \Sigma = f(\overline D)$ is a compact complex 
curve in $\C^2$ with embedded $\cC^1$ boundary $b\Sigma=f(bD)$
and with finitely many interior singularities.
Let $\{d_1,\ldots, d_s\}=f^{-1}(\Sigma_{\rm sing}) \subset D$.

We realize $\overline D$ as a domain with smooth 
boundary in an open Riemann surface $R$; the 
corresponding biholomorphic map is of class $\cC^1$ up
to the boundary. By Mergelyan's theorem we can find a holomorphic map 
$g\colon U\to \C^2$ from an open neighborhood $U\subset R$ of 
$\overline D$ into $\C^2$ such that $g$ approximates $f$
arbitrarily well in the $\cC^1(\overline D)$ topology, and $g$
agrees with $f$ to a given order at each of the points 
$c_1,\ldots,c_l$, $d_1,\ldots,d_s$.
By Lemma \ref{invariance} we may assume that the complex curve
$g(D)$ is biholomorphic to $f(D)$. Replacing $g$ by $f$ and
$R$ by a sufficiently small open neighborhood of
$\overline D$ in $R$ we may therefore assume that $f \colon R\to \C^2$ is a 
holomorphic map which is an embedding (injective immersion)
on $R\bs \{d_1,\ldots,d_s\}$.

We have $bD=\bigcup_{j=1}^m C_j$, each $C_j$ being a closed curve.
For every $j$ we choose a point $a_j\in  C_j$ and a
smooth embedded arc $\gamma_j\subset R$ that is attached with
one of its endpoints to $\overline D$ at $a_j$, and such that the intersection
of $\gamma_j$ and $C_j$ is transverse at $a_j$.
The rest of the arc, $\gamma_j\bs\{a_j\}$, is contained
in $R\bs \overline D$. Let $b_j$ denote the other endpoint of $\gamma_j$.
Choose an open set $U\subset R$ that contains $\overline D$
and  such that $\overline U$ does not contain any of the points
$b_1,\ldots,b_m$. We also insure that the set
$\gamma_j\cap U=\wt \gamma_j$ is an arc with an endpoint $a_j$.

In $\C^2$ we choose for every $j=1,\ldots,m$ a smooth embedded arc
$\lambda_j$ that agrees with the arc $f(\wt \gamma_j)$ near the
endpoint $q_j= f(a_j)$, while the rest of it, $\lambda_j \bs f(\wt
\gamma_j)$, does not intersect $f(U)$. We also insure that the
arcs $\lambda_1,\ldots,\lambda_m$ are pairwise disjoint, 
they do not intersect any of the vertical complex lines through 
the points $f(c_1),\ldots,f(c_l)$, and 
the other endpoint $p_j$ of $\lambda_j$ is an exposed point
for the set $f(\overline D)\cup (\bigcup_{j=1}^m \lambda_j) \subset \C^2$
(see Figure \ref{Fig3}). In particular, the complexified tangent
line to the arc $\lambda_j$ at $p_j$ is transverse to the vertical
line through $p_j$. We may begin with an arbitrary
set of points $p_1,\ldots,p_m\in\C^2$ such that the vertical lines
through them are pairwise disjoint and do not intersect $f(\overline U)$,
and then find arcs $\lambda_j$ from $q_j=f(a_j)$ to $p_j$ as above.

%
%
%
%  Fig. 3:
%
%

\begin{figure}[ht]
\psset{unit=0.6cm, linewidth=0.7pt}

\begin{pspicture}(-6,-5)(8,5)
%
%  zunanji rob D
%
\psarc[linewidth=1.5pt,linecolor=red](0,0){4.5}{80}{280}

%
% Tri robne krivulje
%
\psellipse[linewidth=1.5pt,linecolor=red](0.75,3.36)(0.3,1.1)
\psellipse[linewidth=1.5pt,linecolor=red](0.75,-3.36)(0.3,1.1)
\psellipse[linewidth=1.5pt,linecolor=red](0.75,0)(0.3,1.1)
%
% dva vezna loka
%
\psarc[linewidth=1.5pt,linecolor=red](0.75,1.68){0.6}{90}{270}
\psarc[linewidth=1.5pt,linecolor=red](0.75,-1.68){0.6}{90}{270}

%
% dve luknji za visji rod
%
\psarc[linewidth=1.5pt,linecolor=red](-0.7,1.6){1.8}{140}{220}
\psarc[linewidth=1.5pt,linecolor=red](-0.7,-1.6){1.8}{140}{220}
\psarc[linewidth=1.5pt,linecolor=red](-4,1.6){2}{-27}{27}
\psarc[linewidth=1.5pt,linecolor=red](-4,-1.6){2}{-27}{27}

%
% zgornji lok
%
\psdot[dotsize=4pt](1,3.4)
\psdot[dotsize=4pt](3,3.4)
\pscurve(1,3.4)(1.5,3.6)(2,3.4)(2.5,3.2)(3,3.4)
\psline[linestyle=dotted](3,-4.8)(3,4.8)
\rput(1.4,3.1){$q_1$}
\rput(3.4,3.1){$p_1$}
\rput(2.1,3.8){$\lambda_1$}

%
% srednji lok
%
\psdot[dotsize=4pt](1,0)
\psdot[dotsize=4pt](6,0)
\pscurve(1,0)(2,-0.3)(3,0)(4,0.3)(5,0.1)(6,0)
\psline[linestyle=dotted](6,-4.8)(6,4.8)
\rput(1.4,0.3){$q_2$}
\rput(6.4,0.3){$p_2$}
\rput(3.8,0.7){$\lambda_2$}

%
% spodnji lok
%
\psdot[dotsize=4pt](1,-3.5)
\psdot[dotsize=4pt](4.5,-3.5)
\pscurve(1,-3.5)(2,-3.3)(3,-3.5)(3.7,-3.7)(4.5,-3.5)
\psline[linestyle=dotted](4.5,-4.8)(4.5,4.8)

\rput(1.4,-3.8){$q_3$}
\rput(4.9,-3.8){$p_3$}
\rput(2.1,-2.8){$\lambda_3$}

\end{pspicture}
\caption{A Riemann surface with exposed tails}
\label{Fig3}
\end{figure}
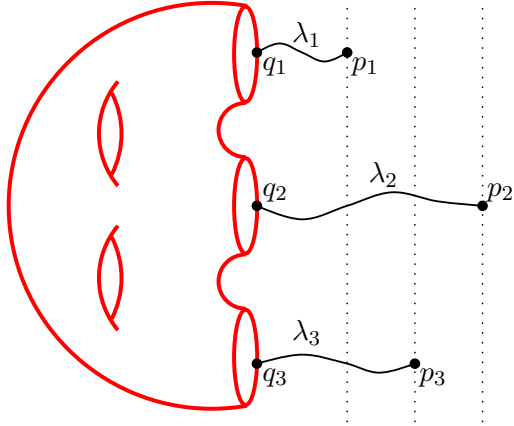

Let $K=\overline D\cup (\bigcup_{j=1}^m \gamma_j)$, a compact set in 
the Riemann surface $R$. 
Let $f'\colon U\cup (\bigcup_{j=1}^m \gamma_j)\to\C^2$ be a smooth map that
agrees with $f$ on $U$ and that maps each arc $\gamma_j\subset R$
diffeomorphically onto the corresponding arc 
$\lambda_j \subset\C^2$. In particular, the endpoint
$b_j$ of $\gamma_j$ is mapped by $f'$ to the exposed endpoint
$p_j$ of $\lambda_j$.

By Mergelyan's theorem (see e.g.\ \cite[Theorem 3.2]{FF:submersions}) we can 
approximate $f'$, uniformly on a neighborhood of $\overline D$ in $R$
and in the $\cC^1$ topology on each of the arcs $\gamma_j$, by
a holomorphic map $\wt f\colon V\to \C^2$ from an open neighborhood
of $K$ in $R$. At the same time we insure that $\wt f$ agrees with $f'$
to a high order at each of the points 
$a_1,\ldots, a_m$, $b_1,\ldots,b_m$, $c_1,\ldots,c_l$, $d_1,\ldots,d_s$. 
If the approximation is close enough, the neighborhood $V\supset K$ 
is chosen small enough, and the interpolation at the indicated
points is to a sufficiently high order, then 
$\wt f\colon V\to \C^2$ is a (non-proper) holomorphic embedding 
except at the points $d_1,\ldots, d_s$, the complex 
curve $\wt \Sigma = \wt f(D) \subset \C^2$ is biholomorphic to 
the curve $\Sigma = f(D)$ according to Lemma \ref{invariance}, and the points 
$p_j=\wt f(b_j)$ and $\wt f(c_j)=f(c_j)$ are exposed in $\wt f(V)$.

Now Theorem \ref{conformal-map} furnishes a diffeomorphism
$\phi\colon \overline D\to \phi(\overline D) \subset V$
that is holomorphic in $D$,  that sends the point $a_j\in bD$ to 
the point $b_j$ for every $j=1,\ldots,m$, that is tangent to the identity
to a desired (high) order at each of the points $c_1,\ldots,c_l$,
$d_1,\ldots, d_s$, and that is close to the identity map outside a 
small neighborhood of $\{a_1,\ldots,a_m\}$. 
The composition
\[
    F=\wt f\circ\phi\colon \overline D \to \C^2
\]
then maps $\overline D$ onto the domain $F(\overline D)$
in the complex curve $\wt f(V) \subset \C^2$
such that each point $p_j=F(a_j)$ for $j=1,\ldots,m$
is an exposed boundary point of $F(\overline D)$, 
and the points $F(c_j)=f(c_j)$ are also exposed
in $F(\overline D)$.

Let $\Sigma'= F(D)$. Note that $\phi$ induces a biholomorphic map
\[
	\wt \phi =  F \circ (\wt f\,)^{-1} 
	     =  \wt f \circ\phi\circ  (\wt f\,)^{-1} 
	        \colon \wt \Sigma_{\rm reg} \to \Sigma'_{\rm reg}.
\]
If $\phi$ is chosen tangent to the identity map to a 
sufficiently high order at each of the points $d_1,\ldots, d_s$,
then $\wt \phi$ is tangent to the identity to a high order
at each of the points in $\wt \Sigma_{\rm sing}$,
and hence Lemma \ref{invariance} shows that $\wt \phi$ extends to 
a biholomorphic map $\Phi\colon \wt \Sigma \to \Sigma'$. 
Thus $\Sigma'=F(D)$ is biholomorphic to $\wt \Sigma$,
and hence to $\Sigma = f(D)$. 
\end{proof}

\section{Proofs of main results}
\label{proof}
In this section we prove Theorem \ref{curves} and obtain some 
further corollaries.

By Theorem \ref{exposing} in \S \ref{sec:exposing} we may
assume that the complex curve $\overline \Sigma$ in Theorem \ref{curves}
admits an exposed point in each of its boundary curves.
To complete the proof of Theorem \ref{curves} it 
therefore suffices to show the following.

%
%
% WOLD'S THEOREM
%
%
\begin{theorem}
\label{Wold}
Let $\overline\Sigma \subset \C^2$ be as in Theorem \ref{curves}.
If every boundary component of $\Sigma$ contains an exposed point
(see Def.\ \ref{exposed}) then the conclusion of Theorem \ref{curves} holds. 
\end{theorem}

\begin{proof}
In the special case when $\Sigma$ has no interior singularities, 
Theorem \ref{Wold} is due to the second author 
(see \cite[Theorem 1]{Wold2}). We shall now show that the proof 
given there also holds for curves with singularities.

Lemma \ref{normalize} furnishes a smoothly bounded domain
$D$ in a Riemann surface $R$ and a $\cC^r$ map
$f \colon \overline D\to \overline \Sigma$ such that
$f(\overline D)=\overline \Sigma$, $f(bD) = b\Sigma$,   
$f$ is diffeomorphic near $bD$, and $f\colon D\to \Sigma$ 
is a holomorphic normalization of $\Sigma$.

Let $bD=\bigcup_{j=1}^m C_j$, and assume that $a_j \in C_j$
is an $f$-exposed point for each $j=1,\ldots,m$
(with respect to the first projection
$\pi_1(z,w)=z$). Let $\pi_2 \colon \C^2\to \C$ be the second 
projection $\pi_2(z,w)=w$. Define a rational shear map $g$ of $\C^2$ by
\begin{equation}
\label{function-g}
    g(z,w)= \bigl( z,w+\sum_{j=1}^m \frac{\alpha_j}{z-\pi(f(a_j))} \bigr).
\end{equation}
The numbers $\alpha_j\in \C\bs\{0\}$ can be chosen such that
$\pi_2$ maps the (unbounded) curves 
\[
	\lambda_j = (g\circ f)(C_j \bs \{a_j\}) \subset\C^2
\]
to unbounded curves $\gamma_j=\pi_2(\lambda_j) \subset \C$,
and $\pi_2\colon \lambda_j\to\gamma_j$ is a diffeomorphism 
near infinity. Furthermore, for every sufficiently large number 
$\rho >0$, the set 
$
	\rho \overline \D \cup \bigcup_{j=1}^m \gamma_j\subset \C 
$
has no bounded complementary connected components.
This is achieved by a careful choice of the arguments of 
$\alpha_j$'s, while their absolute values $|\alpha_j|$ 
can be taken as small as desired.  

Consider the  map 
$g\circ f\colon \overline  D\bs \{a_j\}_{j=1}^m\to \C^2$.
Fix a compact set $L$ in $D$.
By choosing the numbers $\alpha_j$ small enough we insure that 
$g\circ f$ is close to $f$ on $L$.  
The complex curve $X=(g\circ f)(D) \subset \C^2$, with boundary
\[
	bX= (g\circ f)(bD\bs  \{a_j\}_{j=1}^m) =\bigcup_{j=1}^m \lambda_j, 
\]
is then biholomorphic to $\Sigma=f(D)$, and it enjoys the 
following properties:

\smallskip
(1) $X$ admits an exhaustion
$K_1\subset K_2\subset\cdots \subset \bigcup_{j=1}^\infty K_j = X$
by compact sets $K_j$ that are polynomially convex in $\C^2$,
with $(g\circ f)(L)\subset K_1$.

To see this, it suffices to show that any smoothly bounded
compact set $K\subset X$ that is holomorphically convex in $X$
is also polynomially convex in $\C^2$.
Since $\widehat K=\widehat{bK}$ and $b K$
is a union of smooth curves, the set $A= \widehat K\bs  b K$
is an analytic subvariety of $\C^2\bs  bK$
containing $K\bs bK$ (see \cite{Stolz}).
If $A\ne K\bs bK$, then $A$ contains a local extension of $K$
in $X$ near a boundary component of $K$. Hence
$\widehat K$ contains at least one connected component of $X\bs K$,
a contradiction since each of these components is unbounded in $\C^2$.
Thus $\widehat K=K$ as claimed. 

(2) A similar argument shows that for any compact polynomially convex
set $K \subset \C^2\bs bX$, $K\cup K_j$ is also polynomially
convex for all large $j\in\N$. 

(3) For every compact polynomially convex set $K$ contained in 
$\C^2\bs bX$ and for every pair of numbers $\epsilon>0$ (small) 
and $R>0$ (large) there exists a holomorphic automorphism $\phi$ of
$\C^2$ such that 
\[
	\sup_{x\in K} |\phi(x)-x| <\epsilon \ \ 
	{\rm and} \ \ \phi(bX) \subset \C^2\bs R\B. 
\]
(Here $\B$ is the unit ball in $\C^2$.) 
This property of $X$ is invariant under 
holomorphic automorphisms of $\C^2$ as is seen by a 
conjugation argument.

The construction of such $\phi$ can be found in \cite{Wold1} 
(see Lemma 1 and the proof of Theorem 4 in \cite{Wold1}); the main point 
to note here is that the construction depends on the geometric assumptions on the 
curves $\lambda_j$ - it has nothing to do with whether or not $X$ is smooth.

\smallskip

Using properties (1)--(3) we find a sequence of holomorphic automorphisms 
$\Phi_j=\phi_j\circ\phi_{j-1}\circ \cdots\circ\phi_1\in\Aut\C^2$
($j=1,2,\ldots$) carrying $bX$ to infinity and converging on $X$ 
to a proper holomorphic embedding $X\hra \C^2$. 
The inductive step is the following.
Fix $j\in\N$ and assume inductively that 
$\Phi_j(bX) \cap j\overline\B =\emptyset$. 
(This trivially holds for $j=0$ with $\Phi_0=\mathrm{id}$.)
Choose $m_j\in \N$ large enough such that the compact set 
$L_j= j\overline\B \cup \Phi_j(K_{m_j})$ is polynomially convex 
(this is possible by property (2)). 
By property (3) there is for any $\epsilon_j >0$ an automorphism  
$\phi_{j+1} \in\Aut \C^2$ such that
\begin{itemize}
\item $|\phi_{j+1}(x)-x|<\epsilon_j$ for all $x\in L_j$, and
\item $|\phi_{j+1}(x)| > j+1$ for all $x\in \Phi_j(bX)$.
\end{itemize}
Setting $\Phi_{j+1}=\phi_{j+1}\circ \Phi_j$ completes the induction step.

Suitable choices of the sequences $\epsilon_j\searrow 0$ 
and $m_j\nearrow +\infty$ insure that the sequence $\Phi_j\in\Aut\C^2$ converges 
locally uniformly on the domain 
\[
	\Omega =\bigcup_{j=1}^{\infty} \Phi_j^{-1}(j\overline \B)\subset\C^2
\] 
to a biholomorphic map $\Phi\colon\Omega\to\C^2$ onto $\C^2$ (a Fatou-Bieberbach map),
and we have $X\subset \Omega$ and $bX\subset b\Omega$
(see \cite[Proposition 5.1]{F1999}).
The restriction $\varphi=\Phi|_X\colon X\hra \C^2$ is then 
a proper holomorphic embedding of $X$ into $\C^2$.  
Since $X$ is biholomorphic to $\Sigma=f(D)$,  this proves Theorem \ref{Wold}. 
\end{proof}

\begin{proof}[Proof of Corollary \ref{cor1}]
Let $\{a_j\}\subset D$ and $\{b_j\}\subset \C^2$  be discrete sequences
without repetition. If $f\colon \overline D\hra\C^2$ is a holomorphic embedding
such that each boundary component of $D$ admits an $f$-exposed point,
it was proved in \cite[Theorem 3]{KLW} that there is a proper holomorphic
embedding $\varphi\colon D\to\C^2$ such that $\varphi(a_j)=b_j$
for $j=1,2,\ldots$.  By Theorem \ref{exposing} such
an embedding $f\colon \overline D\hra\C^2$ with exposed 
boundary points exists for every Riemann surface $D$ 
satisfying the hypothesis of Corollary \ref{main}. 
\end{proof}

We also have the following embedding result for certain
bordered Riemann surfaces with punctures.

\begin{theorem}
\label{puncture}
Asume that $f\colon\overline D\to \C^2$ is as in Corollary \ref{main},
$\pi\colon \C^2 \to \C$ is a $\C$-linear projection,
$b_1,...,b_k\in\C$, and 
\[
	\{c_1,\ldots,c_l\}= (\pi\circ f)^{-1}(\{b_1,\ldots,b_k\}) \subset D.
\]
Then $D\bs \{c_1,\ldots, c_l\}$ embeds properly holomorphically in $\C^2$.
\end{theorem}

\begin{proof}
By a linear change of coordinates on $\C^2$ we may assume that 
$\pi$ is the first coordinate projection.
Theorem \ref{exposing} furnishes a new embedding
$F\colon \overline D \hra \C^2$ with an exposed point $a_j\in bD$ 
in each boundary component, taking care to insure that
$F(c_j)=f(c_j)$ for $j=1,\ldots,l$. The construction also
shows that we can avoid creating any new intersections
of $F(\overline D)$ with the finitely many complex lines
$\pi^{-1}(b_j)$ for $j=1,\ldots,k$, so that we have 
\[
	(\pi\circ F)^{-1}(\{b_1,\ldots,b_k\}) = \{c_1,\ldots, c_l\} \subset D.
\]
Let $g$ be a shear (\ref{function-g}) with simple poles at all points
$(\pi\circ F)(a_j)$ $(j=1,\ldots,m)$ and 
$b_1,\ldots, b_k$. Then $g\circ F$ embeds the punctured 
domain $D'=D\bs \{c_1,\ldots,c_l\}$ onto a complex curve $X\subset \C^2$.
The rest of the proof (pushing $bX$ to infinity) is exactly as 
in the proof of Theorem \ref{Wold}. 
\end{proof}

%
%
%  DOMAINS WITH PUNCTURES
%
%
\begin{corollary}
\label{cor-punctures}
Assume that the embedding $f \colon \overline D\hra\C^2$
satisfies the hypotheses of Corollary \ref{main}.
If $c_1,\ldots, c_l\in D$ are $f$-exposed points
(with respect to some linear projection $\pi\colon\C^2\to\C$),
then the domain $D'=D\bs \{c_1,\ldots,c_l\}$ admits  
a proper holomorphic embedding in $\C^2$.
\end{corollary}

In particular, every finitely connected planar domain
with finitely many punctures embeds properly in $\C^2$,
a result first proved by Wold \cite{Wold1}
(for the punctured disc see also Alexander \cite{Alexander}
and Globevnik \cite{Gl}).

%
%
%  TEICHMULLER SPACES
%
%
%
\section{Teichm\"uller spaces of bordered Riemann Surfaces.}
\label{Teichmuller} In this section we outline another possible
proof of Corollary \ref{main} by employing the theory of
Teichm\"uller spaces. Although not nearly as explicit as
our main proof, it sheds additional light on the subject.
The main idea was already used by Globevnik and Stens\o nes (see \cite{GS}) 
for planar domains (genus $g=0$), and by the second author (see \cite{Wold3}) 
for domains in complex tori (genus $g=1$). Here we focus on domains of
genus $g>1$.

Let $R$ be a connected, closed, oriented smooth surface of genus $g>1$.
The set of all equivalence classes of complex structures
on $R$ is the quotient $T_g/\varGamma_g$, where $T_g$
is the Teichm\"uller space of $R$ (a complex manifold
of complex dimension $3g-3$ that is biholomorphic to a
bounded domain in $\C^{3g-3}$ and is homeomorphic to the ball),
and $\varGamma_g$ is a properly discontinuous group of holomorphic
automorphisms of $T_g$. (For a precise description and the
construction of the Teichm\"uller space $T_g$
see \cite{Ahlfors0,Ahlfors1,Ahlfors2,Bers1,Bers2,Bers3}
and the monographs \cite{Lehto,LV}.)
Each element of $T_g$ can be represented uniquely as
the quotient $\D/G$ of the unit disc $\D\subset \C$ by a
suitably normalized {\em Fuchsian group} $G$, that is, 
a group of fractional linear transformations preserving
the circle $b\,\D$ and acting properly discontinuously
and without fixed points on both discs forming the complement
of $b\,\D$ in the Riemann sphere $\P^1=\C\cup\{\infty\}$. 
By fixing a marked reference surface
$R_0=\D/G_0 \in T_g$, we may view $T_g$ as the
space of group isomorphisms $\theta\colon G_0\to G$
of normalized Fuchsian groups, with the coefficients of
the generators of $\theta(G_0)$ serving as the coordinates
(see \cite[Theorem 2]{Bers2}).

There exists a holomorphic submersion
$\pi\colon Z\to T_g$ of a complex manifold $Z$ onto the Teichm\"uller space
$T_g$ such that the fiber $\pi^{-1}(\theta)$ over any point
$\theta\in T_g$ is the Riemann surface $R_\theta=\D/\theta(G_0)$;
hence $Z$ is a {\em universal family of closed Riemann surfaces
of genus $g$}. One takes $Z$ as the quotient of
$X=T_{g}\times\D$ obtained by replacing each fiber
$\{\theta\}\times \D \subset X$
by the Riemann surface $\D/\theta(G_0)$.
Ahlfors showed  that, in the complex structure
on $X$, the maps $(\theta,z) \mapsto\theta$ and
$(\theta,z)\mapsto (\theta,\theta(a) z)$ are holomorphic for
a fixed $a\in G_0$ (see \cite{Ahlfors2}), and this gives a complex structure to $Z$.

We now consider connected domains $D \subset R$ obtained by removing
$m\ge 1$ discs (homeomorphic images of the closed disc
$\overline \D$) from $R$. The boundary $bD$ of any such domain
is the union of $m$ closed Jordan curves, each bounding a complementary
disc that was removed from $R$. We shall write $R_\theta$ for the
Riemann surface obtained by endowing $R$ with the complex structure
determined by a point $\theta\in T_g$.

He and Schramm proved (see \cite{HS,HS2}) that every domain
$D\subset R_\theta$ as above is conformally equivalent
to a domain $D'$ in another Riemann surface
$R'= R_{\theta'}$ such that the preimage of $D'$
in the universal covering $\D$ of $R'$ is a domain
in $\D$ all of whose complementary components
are geometric (round) discs; we shall call such $D'$
a {\em circle domain}. Moreover, the operation
mapping $D$ to $D'$ is continuous, in the sense that domains
close to $D$ are mapped to circle domains close to $D'$
in Riemann surfaces close to $R'$.
For connected planar domains with at most countably many boundary components,
this solved a famous conjecture of Koebe from 1908 to the effect that
every planar domain is conformally equivalent to a circle domain.
Known as the {\em Kreisnormierungsproblem}, this conjecture was
the subject of considerable effort over many decades.

Using the result of He and Schramm, one can give the
following description of the Teichm\"uller space
$T_{g,m}$ of bordered Riemann surfaces of genus
$g\ge 2$ with $m\ge 1$ boundary components.
Every element of $T_{g,m}$ is represented by a
circle domain $D$ in a closed Riemann surface $R_\theta$
of genus $g$, determined by a point $\theta \in T_g$.
We represent $D$ by a choice of representatives
$(z,r)=(z_1,\ldots,z_m,r_1,\ldots,r_m) \in \D^m\times (0,\infty)^m$
of the centers $z_j\in \D$ and the radii $r_j>0$ of the complementary
components of the preimage of $D$ in $\D$;
such triples $(\theta,z,r)$ then parametrize the points in $T_{g,m}$.
Although this representation of $D$ is clearly not unique
as we may choose different representatives of the removed discs,
it is locally unique in the following sense:
If $\epsilon>0$ is small enough then
the triples $(\theta',z',r')$ that are $\epsilon$-close
to $(\theta,z,r)$ determine pairwise distinct elements
of $T_{g,m}$. (This is seen by observing that the
Fuchsian group $G=\theta(G_0)$ acts properly discontinuously
and without fixed points on $\D$, and for each removed disc
$\Delta\subset \D$ we also remove all its images $g(\Delta)$
for $g\in\theta$.) In this way we define on
$T_{g,m}$ the structure of a real
$(6g-6+3m)$-dimensional manifold.

Let $E_{g,m}$ denote the set of all circle domains $D$
in Riemann surfaces $R_\theta$ $(\theta\in T_g)$
such that $\overline D$ admits
an injective immersion $f\colon \overline D\hra \C^2$
that is holomorphic in $D$. In other words, $E_{g,m}$
is the set of elements of the Teichm\"uller space $T_{g,m}$
that satisfy the hypothesis of Corollary \ref{main}.

\begin{proposition}
\label{open}
The set $E_{g,m}$ is nonempty and open in $T_{g,m}$.
\end{proposition}

\begin{proof}
That $E_{g,m}$ is nonempty was proved in
\cite[Theorem 1.1]{CF}, and it also follows from our results:
Any compact Riemann surface $R$ admits an immersion to
$\P^2$, and by cutting out a suitably chosen open disc
$U\subset R$ one obtains a holomorphic embedding of
$\overline D_0 = R\bs U$ into $\C^2$.
Removing $m-1$ additional pairwise disjoint closed
discs from $D_0$ we obtain a point in $E_{g,m}$.

To see that $E_{g,m}$ is open, choose a point $(\theta,z,r) \in E_{g,m}$
and let $D\subset R_\theta$ denote the correponding circle domain.
Let $f\colon \overline D \hra \C^2$ be an embedding as in 
Corollary \ref{main}.
We can approximate $f$ in the $\cC^1(\overline D)$ topology 
by a holomorphic map $f\colon U\to \C^2$ from an open set
$U\subset R_\theta$ containing $\overline D$.

Consider $R_\theta$ as the fiber $\pi^{-1}(\theta)$
in the fibration $\pi\colon Z\to T_g$ defined above.
The set $\overline D \subset R_\theta$ admits
an open Stein neighborhood in $R_\theta$
(just remove a point from each connected component
of $R_\theta\bs \overline D$), and hence it has a basis
of open Stein neighborhoods $\Omega\subset Z$
by Siu's theorem \cite[Theorem 1]{Siu}. Choose $\Omega$
small enough such that $\Omega\cap R_\theta \subset U$.
By Cartan's extension theorem, the map $f\colon U\cap \Omega \to \C^2$
extends to a holomorphic map $F\colon \Omega\to\C^2$.
The restriction of $F$ to any domain $\overline D'\subset R_{\theta'}$
sufficiently near $D$ (in a fiber $R_{\theta'}$ of $Z$
that is sufficiently close to the initial fiber $R_\theta$)
is then a holomorphic embedding of $\overline D'$ into $\C^2$,
and hence such $D'$ belongs to $E_{g,m}$.
This completes the proof of Proposition \ref{open}.
\end{proof}

\begin{problem}
\label{closed}
Is the set $E_{g,m}$ closed in $T_{g,m}$ ?
\end{problem}

An affirmative answer would imply 
that every bordered Riemann surface embeds properly
holomorphically into $\C^2$. 
For the time being this seems entirely out of reach.

\begin{proof}[Sketch of an alternative proof of Corollary \ref{main}]
Fix a circle domain $D\subset \R_\theta$
satisfying the hypothesis of the corollary. 
The argument in the proof of 
Proposition \ref{open} above
give a smooth family of holomorphic embeddings
of (the closures of) all nearby circle domains into $\C^2$.
Proposition 3 in \cite{Wold3} gives another continuously
varying family of holomorphically embedded surfaces in $\C^2$,
close to the original one, whose members all have
an exposed point in each boundary component,
and hence they all embed properly holomorphically
into $\C^2$ by Theorem \ref{Wold}. (This construction
of exposed points is reminiscent of what we did
in the proof of Proposition \ref{exposing} above, 
but less precise as it entails
a small cut of each domain, thereby changing
its conformal structure. At this point one must use 
that the normalization provided by He and Schramm
is a continuous operation.)

An argument as in \cite{GS} and \cite{Wold3}, using the Brouwer fixed point
theorem, now shows that there is a domain in the new
family that is conformally equivalent to the original domain $D$,
thereby concluding the proof.  For domains in tori  
the details of this argument can be found in \cite{Wold3}.
\end{proof}

%
%
%  THANKS, THANKS AND THANKS
%
%
%

\smallskip
\textit{Acknowledgements.}
A part of this work was done when the authors were visiting the 
Department of Mathematics of the University of Berne in the Spring of 2007. 
We wish to thank this institution, and in particular Frank Kutzschebauch, 
for the kind invitation and hospitality, and for providing excellent 
working conditions. The second author wishes to thank the Mittag-Leffler 
Institute in Djursholm for the hospitality and support during a
revision of the paper in February 2008, and J.-P.\ Rosay for having proposed a proof 
of Lemma \ref{weakly-holo} by using the $\dibar$-equation.

\bibliographystyle{amsplain}

\begin{thebibliography}{10}


\bibitem{Ahlfors0}
L.\ V.\ AHLFORS,
On quasiconformal mappings,
\textit{J.\ d'Analyse Math.} \textbf{3} (1953/54), 1--58.

\bibitem{Ahlfors1}
L.\ V.\ AHLFORS,
`Teichm\"uller spaces' in
\textit{Proc.\ Internat.\ Congress Math.\ (Stockholm, 1962)},
3--9, Inst.\ Mittag-Leffler, Djursholm, 1963.

\bibitem{Ahlfors2}
L.\ V.\ AHLFORS,
Some remarks on Teichm\"uller's space on Riemann surfaces,
\textit{Ann.\ of Math.}\ (2) \textbf{74} (1961), 171--191.

\bibitem{Ahlfors-Bers}
L.\ V.\ AHLFORS and L.\ BERS,
Riemann's mapping theorem for variable metrics,
\textit{Ann.\ of Math.}\ (2) \textbf{72} (1960), 385--404.

\bibitem{AB}
L.\ V.\ AHLFORS and A.\ BEURLING,
The boundary correspondence under quasiconformal mappings,
\textit{Acta Math.}\ \textbf{96} (1956), 125--142.

\bibitem{Ahlfors-Sario}
L.\ V.\ AHLFORS and L.\ SARIO,
\textit{Riemann Surfaces},
Princeton Mathematical Series \textbf{26},
Princeton University Press, Princeton, N.J., 1960.

\bibitem{Alexander}
H.\ ALEXANDER,
Explicit imbedding of the (punctured) disc into $\C^2$,
\textit{Comment.\ Math.\ Helv.}\ \textbf{52} (1977), 539--544.

\bibitem{AL}
E.\ ANDERS\'EN and L.\ LEMPERT,
On the group of holomorphic automorphisms of $\C^n$,
\textit{Inventiones Math.}\ \textbf{110} (1992), 371--388.

\bibitem{BN}
S.\ R.\ BELL and R.\ NARASIMHAN,
`Proper holomorphic mappings of complex spaces'
in \textit{Several complex variables, VI},
Encyclopaedia Math. Sci.\ \textbf{69}, 1--38,
Springer-Verlag, Berlin, 1990.

\bibitem{Bers0}
L.\ BERS,
`Spaces of Riemann surfaces' in
\textit{Proc.\ Internat.\ Congress Math.\ (Mexico City, 1958)}, 349--361,
Cambridge Univ.\ Press, New York, 1960.

\bibitem{Bers1}
L.\ BERS,
Simultaneous uniformization,
\textit{Bull.\ Amer.\ Math.\ Soc.}\ \textbf{66} (1960), 94--97.

\bibitem{Bers2}
L.\ BERS,
Spaces of Riemann surfaces as bounded domains,
\textit{Bull.\ Amer.\ Math.\ Soc.}\ \textbf{66} (1960), 98--103.

\bibitem{Bers3}
L.\ BERS,
The equivalence of two definitions of quasiconformal mappings,
\textit{Comment.\ Math.\ Helv.}\ \textbf{37} (1962/1963), 148--154.

\bibitem{Bishop}
E.\ BISHOP,
Mappings of partially analytic spaces,
\textit{Amer.\ J.\ Math.}\ \textbf{83} (1961), 209--242.

\bibitem{BF}
G.\ BUZZARD and F.\ FORSTNERI\v C,
A Carleman type theorem for proper holomorphic embeddings,
\textit{Ark.\ Mat.}\ \textbf{35} (1997), 157--169.

\bibitem{CF}
M.\ \v CERNE  and F.\ FORSTNERI\v C,
Embedding some bordered Riemann surfaces in the affine plane,
\textit{Math.\ Res.\ Lett.}\ \textbf{9} (2002), 683--696.

\bibitem{CG}
M.\ \v CERNE and J.\ GLOBEVNIK,
On holomorphic embedding of planar domains into $\C^2$,
\textit{J.\ d'Analyse Math.}\ \textbf{8} (2000), 269--282.

\bibitem{Chirka}
E.\ M.\ CHIRKA, \textit{Complex analytic sets}, 
Kluwer, Dordrecht, 1989.

\bibitem{EG}
Y.\ ELIASHBERG and M.\ GROMOV,
Embeddings of Stein manifolds,
\textit{Ann.\ of Math.}\ \textbf{136} (1992), 123--135.

\bibitem{FK}
H.\ M.\ FARKAS and I.\ KRA,
\textit{Riemann Surfaces},
Second ed.
Graduate Texts in Mathematics \textbf{71},
Springer-Verlag, New York, 1992.


\bibitem{F1999}
F.\ FORSTNERI\v C,
Interpolation by holomorphic automorphisms and embeddings in $\C^n$,
\textit{J.\ Geom.\ Anal.}\ \textbf{9} (1999), 93--118.

\bibitem{F:Acta}
F.\ FORSTNERI\v C,
Noncritical holomorphic functions on Stein manifolds,
\textit{Acta Math.}\ \textbf{191} (2003), 143--189.

\bibitem{FF:submersions}
F.\ FORSTNERI\v C,
Holomorphic submersions from Stein manifolds,
\textit{Ann.\ Inst.\ Fourier} \textbf{54} (2004),  1913--1942.

\bibitem{FR}
F.\ FORSTNERI\v C and J.-P.\ ROSAY,
Approximation of biholomorphic mappings by automorphisms of $\C^n$,
\textit{Invent.\ Math.}\ \textbf{112} (1993), 323--349.
Erratum, \textit{Invent.\ Math.}\ \textbf{118} (1994), 573--574.


\bibitem{Gouma}
T.\ GOUMA,
Ahlfors functions on non-planar Riemann surfaces whose double are hyperelliptic,
\textit{J.\ Math.\ Soc.\ Japan.}\ \textbf{50} (1998), 685--695.

\bibitem{Gl}
J.\ GLOBEVNIK,
Interpolation by proper holomorphic embeddings of the disc into $\C^2$,
\textit{Math.\ Res.\ Lett.}\ \textbf{9} (2002), 567--577.

\bibitem{GS}
J.\ GLOBEVNIK and B.\ STENS\O NES,
Holomorphic embeddings of planar domains into $\C^2$,
\textit{Math.\ Ann.}\ \textbf{303} (1995), 579--597.

\bibitem{Goluzin}
G.\ M.\ GOLUZIN,
\textit{Geometric Theory of Functions of a Complex Variable},
Translations of Mathematical Monographs \textbf{68},
American Mathematical Society, Providence, R.I., 1969.

\bibitem{Griffiths-Harris}
P.\ A.\ GRIFFITHS and J.\ HARRIS,
\textit{Principles of Algebraic Geometry},
John Wiley \& Sons, Inc., New York, 1978 and 1994.

\bibitem{GR}
R.\ C.\ GUNNING and H.\ ROSSI,
\textit{Analytic Functions of Several Complex Variables},
Prentice--Hall, Englewood Cliffs, 1965.

\bibitem{HS}
Z.-X.\ HE and O.\ SCHRAMM,
Fixed points, Koebe uniformization and circle packings,
\textit{Ann.\ Math.}\ (2) \textbf{137} (1993), 369--406.

\bibitem{HS2}
Z.-X.\ HE and O.\ SCHRAMM,
Koebe uniformization for "almost circle domains",
\textit{Amer.\ J.\ Math.}\ \textbf{117} (1995), 653--667.

\bibitem{HW}
L.\ H\"ORMANDER and J.\ WERMER,
Uniform approximations on compact sets in $\C^n$,
\textit{Math.\ Scand.}\ \textbf{23} (1968), 5--21.

\bibitem{KLW}
F.\ KUTZSCHEBAUCH, E.\ L\O W and E.\ F.\ WOLD,
Embedding some Riemann surfaces into $\C^2$ with interpolation,
\textit{Math.\ Z.}, to appear.

\bibitem{Laufer}
H.\ B.\ LAUFER,
Imbedding annuli in $\C^2$,
\textit{J.\ Analyse Math.}\ \textbf{26} (1973), 187--215.

\bibitem{Lehto}
O.\ LEHTO,
\textit{Univalent Functions and Teichm\"uller spaces},
Graduate Texts in Mathematics \textbf{109},
Springer-Verlag, New York, 1987.

\bibitem{LV}
O.\ LEHTO and K.\ I.\ VIRTANEN,
\textit{Quasiconformal Mappings in the Plane},
Springer-Verlag, New York, 1973.

\bibitem{Narasimhan}
R.\ NARASIMHAN,
Imbedding of holomorphically complete complex spa\-ces,
\textit{Amer.\ J.\ Math.}\ \textbf{82} (1960), 917--934.

\bibitem{Pom}
Ch.\ POMMERENKE,
\textit{Boundary Behaviour of Conformal Maps},
Springer-Verlag, Berlin-Heidelberg, 1992.

\bibitem{Remmert}
R.\ REMMERT,
Sur les espaces analytiques holomorphiquement s\'epa\-rab\-les et
holomorphiquement convexes,
\textit{C.\ R.\ Acad.\ Sci.\ Paris}\ \textbf{243} (1956), 118--121.

\bibitem{Rudin}
W.\ RUDIN,
Pairs of inner functions on finite Riemann surfaces,
\textit{Trans. \ Amer. \ Math. \ Soc.}\ \textbf{140} (1969), 423--434.

\bibitem{Sch}
J.\ SCH\"URMANN,
Embeddings of Stein spaces into affine spaces of minimal dimension,
\textit{Math.\ Ann.}\ \textbf{307} (1997), 381--399.

\bibitem{Siu}
Y.-T.\ SIU,
Every Stein subvariety admits a Stein neighborhood,
\textit{Invent.\ Math.}\ \textbf{38} (1976/77), 89--100.

\bibitem{Spr} 
G.\ SPRINGER, 
\textit{Introduction to Riemann surfaces},
Addison-Wesley, Reading, Mass., 1957.

\bibitem{Stehle}
J.-L.\ STEHL\'E,
`Plongements du disque dans $\C^2$' in
\textit{S\'eminaire P.\ Lelong (Analyse)},
Lect.\ Notes in Math.\ \textbf{275}, 119--130,
Springer-Verlag, Berlin--New York, 1970.

\bibitem{Stolz}
G.\ STOLZENBERG,
Uniform approximation on smooth curves,
\textit{Acta Math.} \textbf{115} (1966), 185--198.

\bibitem{Whitney}
H.\ WHITNEY,
\textit{Complex analytic varieties}, 
Addison-Wesley, Reading, Mass.-London-Don Mills, Ontario, 1972.

\bibitem{Wold1}
E.\ F.\ WOLD,
Proper holomorphic embeddings of finitely and some infinitely
connected subsets of $\C$ into $\C^2$,
\textit{Math.\ Z.}\ \textbf{252} (2006), 1--9.

\bibitem{Wold2}
E.\ F.\ WOLD,
Embedding Riemann surfaces properly into $\C^2$,
\textit{Internat.\ J.\ Math.}\ \textbf{17} (2006), 963--974.

\bibitem{Wold3}
E.\ F.\ WOLD,
Embedding subsets of tori properly into $\C^2$,
\textit{Ann.\ Inst.\ Fourier} (2007).

\end{thebibliography}

\end{document}